 \documentclass[letterpaper,11pt]{article}
\usepackage{etoolbox}
\newtoggle{anonymous}
\togglefalse{anonymous}

\usepackage[usenames,dvipsnames,svgnames,table]{xcolor}
\usepackage[colorlinks=true,linkcolor=blue,citecolor=ForestGreen]{hyperref}
\usepackage{amsmath,amsthm,amssymb,enumerate}
\usepackage{scrextend}
\usepackage[utf8]{inputenc}
\usepackage[T1]{fontenc}
\usepackage{pgfplots}
\usepackage{paralist}
\usepackage{enumitem}

\usepackage{amsmath}
\usepackage{amssymb}
\usepackage{bm}
\usepackage{thmtools}
\usepackage{hyperref}
\usepackage[capitalize]{cleveref}
\usepackage{float}
\usepackage[color=green!40,textsize=small]{todonotes}
\usepackage{xcolor}
\usepackage{forest}

 \usepackage[margin=1in]{geometry}

\usepackage{setspace}
\allowdisplaybreaks

\Crefname{figure}{Figure}{Figures}
\Crefname{claim}{Claim}{Claims}

\crefformat{equation}{\textup{#2(#1)#3}}
\crefrangeformat{equation}{\textup{#3(#1)#4--#5(#2)#6}}
\crefmultiformat{equation}{\textup{#2(#1)#3}}{ and \textup{#2(#1)#3}}
{, \textup{#2(#1)#3}}{, and \textup{#2(#1)#3}}
\crefrangemultiformat{equation}{\textup{#3(#1)#4--#5(#2)#6}}%
{ and \textup{#3(#1)#4--#5(#2)#6}}{, \textup{#3(#1)#4--#5(#2)#6}}{, and \textup{#3(#1)#4--#5(#2)#6}}

\Crefformat{equation}{Equation~\textup{#2(#1)#3}}
\Crefrangeformat{equation}{Equations~\textup{#3(#1)#4--#5(#2)#6}}
\Crefmultiformat{equation}{Equations~\textup{#2(#1)#3}}{ and \textup{#2(#1)#3}}
{, \textup{#2(#1)#3}}{, and \textup{#2(#1)#3}}
\Crefrangemultiformat{equation}{Equations~\textup{#3(#1)#4--#5(#2)#6}}%
{ and \textup{#3(#1)#4--#5(#2)#6}}{, \textup{#3(#1)#4--#5(#2)#6}}{, and \textup{#3(#1)#4--#5(#2)#6}}

\usepackage{tikz}
\tikzset{normalnode/.style={circle, draw, fill=black, inner sep=0, minimum width=1.5mm}}

\usetikzlibrary{fit}
\usetikzlibrary{shapes,arrows,shadows,positioning}
\usetikzlibrary{backgrounds}
\usetikzlibrary{calc}

\definecolor{agreen}{rgb}{.24,.86,.53}

\newcommand{\Nn}{\mathbb{N}}

\newcommand{\Fc}{\mathcal{F}}

\newcommand{\Hc}{\mathcal{H}}

\newcommand{\Nc}{\mathcal{N}}
\newcommand{\Pc}{\mathcal{P}}
\newcommand{\Uc}{\mathcal{U}}
\newcommand{\Qc}{\mathcal{Q}}
\newcommand{\Xc}{\mathcal{X}}

\newcommand{\Sc}{\mathcal{S}}
\newcommand{\Cc}{\mathcal{C}}
\newcommand{\Dc}{\mathcal{D}}

\newcommand{\sub}{\#\mathrm{Sub}}
\newcommand{\emb}{\#\mathrm{Emb}}

\newcommand{\dt}[2]{#1\text{-}\operatorname{subd}(#2)}
\newcommand{\lab}[1]{\operatorname{Lab}(#1)}

\newtheorem{theorem}{Theorem}[section]

\newtheorem{proposition}[theorem]{Proposition}

\newtheorem{corollary}[theorem]{Corollary}
\newtheorem{observation}[theorem]{Observation}
\newtheorem{conjecture}[theorem]{Conjecture}

\newtheorem{lemma}[theorem]{Lemma}

\theoremstyle{definition}

\theoremstyle{remark}

\newtheorem{claim}[theorem]{Claim}

\newenvironment{poc}{\begin{proof}[Proof of {C}laim.]}{\end{proof}}

\newcommand{\aut}{\mathrm{aut}}

\newcommand{\ctg}{\mathrm{ctg}}

\renewcommand{\leq}{\leqslant}
\renewcommand{\geq}{\geqslant}

\newcommand{\Ex}[1]{\mathbb{E} \left[\,#1\,\right]}
 
\newcommand{\comment}[1]{}

\title{Adjacency Labeling Schemes for Small Classes}

\date{}

 \iftoggle{anonymous}{% ANONYMOUS
 	\author{ }
 }{%
	\author{{\'E}douard Bonnet\thanks{Univ.~Lyon, ENS de Lyon, UCBL, CNRS, LIP, France,  \texttt{edouard.bonnet@ens-lyon.fr}}    \and Julien Duron\thanks{Univ.~Lyon, ENS de Lyon, UCBL, CNRS, LIP, France, \texttt{julien.duron@ens-lyon.fr}} \and    John Sylvester\thanks{Department of Computer Science, University of Liverpool, UK, \texttt{john.sylvester@liverpool.ac.uk}}
	\and
	Viktor Zamaraev\thanks{Department of Computer Science, University of Liverpool, UK, \texttt{viktor.zamaraev@liverpool.ac.uk}} }
}

\begin{document}
\maketitle

\begin{abstract}
  A~graph class admits an \emph{implicit representation}\footnote{To simplify and lighten the abstract, we give the definition of \emph{implicit representations} for \emph{factorial} classes. The reader will find the general definition in the introduction.} if, for every positive integer $n$, its $n$-vertex graphs have a~$b(n)$-bit (adjacency) labeling scheme with $b(n)=O(\log n)$, i.e., their vertices can be labeled by binary strings of length $b(n)$ such that the presence of an edge between any pair of vertices $u, v$ can be deduced solely from the labels of~$u$ and $v$.
  The famous Implicit Graph Conjecture posited that every \emph{hereditary} (i.e.,~closed under taking induced subgraphs) \emph{factorial} (i.e.,~containing $2^{O(n \log n)}$ $n$-vertex graphs) class admits an implicit representation.
  The conjecture was finally refuted [Hatami and Hatami, FOCS '22], and does not even hold among \emph{monotone} (i.e.,~closed under taking subgraphs) factorial classes [Bonnet et al., ICALP '24].
  However, monotone \emph{small} (i.e., containing at most $n! c^n$ many $n$-vertex graphs for some constant~$c$) classes do admit implicit representations.

  This motivates the \emph{Small Implicit Graph Conjecture}: Every hereditary small class admits an~$O(\log n)$-bit labeling scheme.
  We provide evidence supporting the Small Implicit Graph Conjecture.
  First, we show that every small \emph{weakly sparse} (i.e., excluding some fixed bipartite complete graph as a subgraph) class has an implicit representation.
  This is a~consequence of the following fact of independent interest proven in the paper:
  Every weakly sparse small class has bounded expansion (hence, in particular, bounded degeneracy).
  The latter generalizes and strengthens the previous results that every monotone small class has bounded degeneracy [Bonnet et al., ICALP '24], and that every weakly sparse class of bounded twin-width has bounded expansion [Bonnet et al., Combinatorial Theory '22].
  Second, we show that every hereditary small class admits an $O(\log^3 n)$-bit labeling scheme, which provides a~substantial improvement over the best-known polynomial upper bound of $n^{1-\varepsilon}$ on the size of adjacency labeling schemes for such classes.
  To do so, we establish that every small class has \emph{neighborhood complexity} $O(n \log n)$, also of independent interest.
  We then apply a~classic result, due to Welzl [SoCG '88], on efficiently ordering the universe of a~set system of low Vapnik-Chervonenkis dimension such that every set can be described as the union of a~limited number of intervals along this order.  
\end{abstract}

%\newpage
%\tableofcontents

\newpage

%%%%%%%%%%%%%%%%%%%%%%%%%%%%%%%%%%% 
\section{Introduction}
\label{sec:intro}
%%%%%%%%%%%%%%%%%%%%%%%%%%%%%%%%%%% 
 
A \emph{class} of graphs is a~set of graphs which is closed under isomorphism. 
For a~class of graphs $\Xc$ we denote by $\Xc_n$ the set of graphs in $\Xc$ with vertex set $[n]$. 
Let $\Xc$ be a~class of graphs and $b : \Nn \rightarrow \Nn$ be a~function.
A~\emph{$b(n)$-bit adjacency labeling scheme} (or simply \emph{$b(n)$-bit labeling scheme})  for $\Xc$ is a~pair (encoder, decoder) of algorithms where for any $n$-vertex graph $G\in \Xc_n$ the encoder assigns binary strings, called \emph{labels}, of length $b(n)$ to the vertices of $G$ such that the adjacency between any pair of vertices can be inferred by the decoder only from their labels.
We note that the decoder depends on the class $\Xc$, but not on the graph~$G$.
The function~$b(\cdot)$ is the \emph{size} of the labeling scheme.
Adjacency labeling schemes were introduced by Kannan, Naor, and Rudich \cite{KNR88,KNR92},
and independently by Muller \cite{Muller88} in the late 1980s and have been actively studied since then.

The binary word, obtained by concatenating labels of the vertices of a~graph $G \in \Xc_n$ assigned by an adjacency labeling scheme, uniquely determines graph $G$.
Thus, a~$b(n)$-bit labeling scheme cannot represent more than $2^{n b(n)}$ graphs on $n$ vertices, and therefore, if $\Xc$ admits a~$b(n)$-bit labeling scheme, then $|\Xc_n| \leq 2^{n b(n)}$.
This implies a~lower bound of $\frac{\log |\Xc_n|}{n}$ on the size $b(n)$ of any adjacency labeling scheme for $\Xc$.
We say that a~graph class $\Xc$ admits an \emph{implicit representation}, if it admits an information-theoretic \emph{order optimal} adjacency labeling scheme, i.e.,\ if $\Xc$ has a~$b(n)$-bit labeling scheme, where $b(n) = O(\log |\Xc_n|/n)$. 
A natural and important question is: which classes admit an adjacency labeling scheme of a~size that matches this information-theoretic lower bound?

Of particular interest is the case of adjacency labeling schemes of size $O(\log n)$.
This is because, under the natural assumption that vertices get assigned pairwise distinct labels, $\lceil \log n \rceil$ is a~lower bound on the size of any labeling
scheme. Thus, understanding the expressive power of the labeling schemes of size of order $\log n$ is a~natural question. 

By the above discussion adjacency labeling schemes of size $O(\log n)$ can only exist for (at most) \emph{factorial} graph classes, i.e., classes $\Xc$ in which the number $|\Xc_n|$ of $n$-vertex graphs grows not faster than $2^{\Theta(n \log n)}$. It is known that the latter condition alone is not enough to guarantee adjacency labels of size $O(\log n)$ \cite{Muller88}. Kannan, Naor, and Rudich \cite{KNR88} asked if the extra restriction on $\Xc$ of being \emph{hereditary}, i.e., closed under taking induced subgraphs, would be enough for the existence of $O(\log n)$-bit adjacency labeling scheme. 
This question was later stated by Spinrad \cite{Spi03} in the form of a~conjecture, reiterated by  Scheinerman  \cite[Chapter~6]{GraphBook}, that became known as the \emph{Implicit Graph Conjecture}.

\begin{labeling}{(\textit{IGC}):}
	\item[(\textit{IGC}):] Any hereditary factorial graph class admits an $O(\log n)$-bit labeling scheme.
\end{labeling}
 
The conjecture remained open for three decades until it was recently refuted in a~strong form by Hatami and Hatami, who showed that for any $\delta$ there exists a~hereditary factorial class that does not admit an $(n^{1/2 - \delta})$-bit labeling scheme \cite{HH22}. This refutation leaves wide open the question
of characterizing hereditary graph classes that admit $O(\log n)$-bit labeling schemes and no plausible conjectural dichotomy is currently available.

It was recently shown that IGC does not hold even in the family of \emph{monotone} graph classes, the hereditary classes which are closed under taking (not necessarily induced) subgraphs \cite{BDSZZ23mono}. On the positive side, it was shown in the same work that IGC holds for every monotone \emph{small} class, i.e., a~monotone class $\Xc$ with $|\Xc_n| \leq n! \cdot c^n$ for some constant $c > 0$. 

The hereditary small classes form a~subfamily of hereditary factorial classes that contains many classes of practical or theoretical importance. For instance, forests \cite{TreesOtter}, planar graphs \cite{Tut62}, classes of bounded treewidth \cite{BP69}, proper minor-closes classes \cite{NSTW06}, unit interval graphs \cite{Han82}, classes of bounded clique-width \cite{ALR09}, and more generally, classes of bounded twin-width \cite{BGK22} are all small.
All of these classes are known to admit order-optimal, and, sometimes, even asymptotically optimal adjacency labeling schemes.
For example, forests admit a~$(\log n + O(1))$-bit labeling scheme \cite{ADK17}, planar graphs admit a~$(1+o(1))\log n$-bit labeling scheme \cite{DEGJMM21}, classes of bounded treewidth admit a~$(1+o(1))\log n$-bit labeling scheme \cite{GL07}, proper minor-closed classes admit a~$(2+o(1))\log n$-bit labeling scheme \cite{GL07}, graphs of clique-width at most $k$ admit a~$O(k \log k \cdot \log n)$-bit labeling scheme \cite{Spi03,Banerjee22}, and graphs of twin-width at most $k$ admit a~$2^{2^{O(k)}} \cdot \log n$-bit labeling scheme \cite{BGK22}. 
It is known that not every hereditary (and even monotone) small class admits an asymptotically optimal labeling scheme of size $(1+o(1)) \log n$ \cite{BDSZZ23}.
However, this does not rule out the existence of order optimal labeling schemes of size $O(\log n)$ for all hereditary small classes.
Alon \cite{Alon23} recently showed that every hereditary class $\Xc$ with $|\Xc_n| = 2^{o(n^2)}$ (which include all hereditary small classes) admits an adjacency labeling scheme of size $n^{1-\varepsilon}$ for some $\varepsilon > 0$ depending on the class.
To our knowledge, no better upper bound was established for hereditary small classes.

The importance of hereditary small classes and the existence of $O(\log n)$-bit labeling schemes for all monotone small classes motivated the formulation of the \emph{Small Implicit Graph Conjecture}.

\begin{labeling}{(\textit{Small IGC}):}
	\item[(\textit{Small IGC}) \cite{BDSZZ23mono}:] Any hereditary small graph class admits an $O(\log n)$-bit labeling scheme.
\end{labeling}

%%%%%%%%%%%%%%%%%%%%%%%%%%%%%%%%%%% 
\subsection{Our contribution} 
%%%%%%%%%%%%%%%%%%%%%%%%%%%%%%%%%%% 

In this paper we provide evidence toward the Small Implicit Graph Conjecture in two independent directions.
First, we show that the Small IGC holds in the important special case of weakly sparse graph classes.
Secondly, we obtain an adjacency labeling scheme of polylogarithmic size for any hereditary small class substantially improving upon the best-known polynomial-size upper bound on adjacency labeling schemes for such classes.

Our first result is a~consequence of a~structural property of small classes of graphs.
To state this property, we need to introduce some terms.
A hereditary class $\Xc$ is said to be \emph{weakly sparse} if there exists an integer $t$ such that the complete bipartite graph $K_{t, t}$ is not a~subgraph of any graph of $\Xc$.
The \emph{degeneracy} of a~graph~$G$ is the minimum number $k$ such that every induced subgraph of~$G$ contains a~vertex of degree at most $k$.
We defer a formal definition of \emph{bounded expansion} until \cref{sec:prelim} as it is somewhat technical.
Instead, we note that \emph{bounded expansion} generalizes \emph{bounded degeneracy}, thus \cref{thm:small-ws-deg} a fortiori holds when the former is replaced by the latter.

%Our first technical result provides structural insight on weakly sparse small graph classes.
\begin{restatable}{theorem}{smallweakdgen}\label{thm:small-ws-deg}
	Every weakly sparse small class has bounded expansion. 
\end{restatable}

This structural insight into small graph classes is of independent interest due to the significance of weakly sparse graph classes.
The notion of weakly sparseness is the broadest form of sparsity.
It generalizes the properties of bounded degree, bounded degeneracy, and nowhere denseness~\cite{sparsity}.
By the celebrated K\H{o}v\'ari--S\'os--Tur\'an theorem \cite{Kovari54}, among hereditary classes, weakly sparse classes are precisely the classes of graphs with a~truly subquadratic number of edges.
It is often observed that the extra restriction of being weakly sparse significantly simplifies the structure of graphs possessing a~particular property.
For example, weakly sparse graph classes of bounded clique-width have bounded treewidth \cite{GurskiW00}; weakly sparse graph classes of bounded twin-width have bounded expansion, and thus bounded degeneracy \cite{BGK22} (a~fact that \cref{thm:small-ws-deg} generalizes, as it was proven in the same paper that every class of bounded twin-width is small); weakly sparse string graphs have bounded degeneracy \cite{Fox14}, and even polynomial expansion~\cite{DvorakN16}; weakly sparse graph classes that are well-quasi-ordered by the induced-subgraph relation have bounded pathwidth \cite{ALR19}, etc. 

%There is a~conjecture from 2016~\cite{Warwick16} (and an ongoing program dedicated to settle it) 
A~conjecture from 2016~\cite{Warwick16}, driving an ongoing program in algorithmic graph theory, suggests
that, under some widely-believed complexity-theoretic assumption, a~hereditary graph class admits a~fixed-parameter tractable (FPT) first-order (FO) model-checking algorithm if and only if it is monadically dependent.\footnote{As with the other side remarks, we do not define the technical terms in this sentence since they will not reappear.}
It was recently proven that every hereditary small class is monadically dependent~\cite{Dreier24b} (and, in the same paper, the \emph{only if} part of the conjecture was established).
Thus, if the conjecture holds, it should in particular be true that every hereditary small class admits an FPT FO model-checking algorithm.
\cref{thm:small-ws-deg} confirms that this is indeed the case within weakly sparse classes, as such an algorithm is known to exist for classes of bounded expansion~\cite{DvorakKT13}.

Dvoř{\'{a}}k and Norine proved that if the expansion of a~class is upper bounded by a~sufficiently slowly-growing subexponential function, then the class is small~\cite{DvorakN10}.
From the definition of \emph{expansion} given in the next section, it should be clear that the converse cannot hold since the small classes consisting of all complete graphs or of all bipartite complete graphs have unbounded expansion.
A~minimum requirement to establish a~partial converse is thus that the class is weakly sparse.
\cref{thm:small-ws-deg} confirms that under this condition the expansion of any small class is indeed bounded (albeit not necessarily by a~subexponential nor a single-exponential function).

We apply \cref{thm:small-ws-deg} to obtain our first main result that the Small IGC holds for weakly sparse graph classes. This follows from 
a~classical labeling scheme for classes of bounded degeneracy \cite{KNR88} (see also \cite[Lemma 3.5]{BDSZZ23mono}).

\begin{theorem}\label{thm:ws-ls}
	Every weakly sparse small class admits an $O(\log n)$-bit adjacency labeling scheme.
\end{theorem}

Our second main result holds more generally for any hereditary small class and goes some way toward the full strength of the Small IGC.  

 \begin{restatable}{theorem}{main}\label{thm:overall-main}
	Every hereditary small class admits an $O(\log^3 n)$-bit adjacency labeling scheme.
\end{restatable}

This result is based on our second structural result of independent interest that hereditary small graph classes have low neighborhood complexity.

 \begin{restatable}{theorem}{neighborhoodComplexity}\label{thm:neighborhood-complexity}
	The neighborhood complexity of any hereditary  small  graph class is $O(n \log n)$.
\end{restatable}

The neighborhood complexity is a~measure of structural complexity of graphs. Informally (a~formal definition is given in \cref{sec:prelim}), it reflects the diversity of the set system of the vertex neighborhoods in the graph.
Low neighborhood complexity of graph classes implies nice structural and algorithmic properties. For example, in the universe of monotone graph classes, linear neighborhood complexity characterizes classes of bounded expansion \cite{RVS19}, and almost linear neighborhood complexity characterizes nowhere dense classes \cite{EGKKP16}.
Low neighborhood complexity is also used to obtain FPT algorithms for various graph problems (see for instance~\cite{EGKKP16, BKRTW22}), and was recently used as a~central ingredient of an FPT algorithm for the FO model-checking problem on monadically stable graph classes \cite{DEMMPT23}. 
It is known that every hereditary graph class of bounded twin-width has linear neighborhood complexity \cite{BKRTW22} (see also \cite{BFLP24} for improved bounds).
We omit formal definitions and result statements, and refer the reader to the respective work.

We obtain our \cref{thm:overall-main} by combining \cref{thm:neighborhood-complexity} with some known result from the theory of sets systems about paths with low crossing number. In a~standard form the latter result is not suitable for our need, and therefore we reformulate it and provide a~self-contained proof.
We believe that the result stated in this form can be of use for other needs.

%%%%%%%%%%%%%%%%%%%%%%%%%%%%%%%%%%% 
\subsection{Organization}	
%%%%%%%%%%%%%%%%%%%%%%%%%%%%%%%%%%% 

The paper is organized as follows.
In \cref{sec:weakly-sparse}, we prove that weakly sparse small classes have bounded expansion (and thus bounded degeneracy); from this, we derive an $O(\log n)$-bit adjacency labeling scheme for such classes.
In \cref{sec:neighborhood-complexity}, we show that every hereditary small class has neighborhood complexity $O(n \log n)$. 
In \cref{sec:neigh-complexity-and-contiguity}, we use this result together with a~classical result from the theory of set systems stated in a~suitable form, to obtain $O(\log^3 n)$-bit labeling schemes for all hereditary small classes.
In \cref{sec:app-paths-low-crossing}, we provide a~self-contained proof of the required form of the result from the theory of set systems.
In the next section, we introduced necessary notation and state auxiliary facts.

%%%%%%%%%%%%%%%%%%%%%%%%%%%%%%%%%%% 
\section{Preliminaries}
\label{sec:prelim}
%%%%%%%%%%%%%%%%%%%%%%%%%%%%%%%%%%% 

\paragraph{Graphs.} We denote by $V(G)$ and $E(G)$ the vertex set and edge set of a~graph~$G$, respectively. When we refer to an $n$-vertex graph $G$ as \emph{labeled}, we mean that the vertex set of $G$ is $[n]$, and we distinguish two different labeled graphs even if they are isomorphic. 
For a~set of graphs $X$, we denote by $\lab{X}$ the set of all labeled graphs isomorphic to a~graph in $X$.
A graph $H$ is an \emph{induced subgraph} of~$G$ (resp.~\emph{subgraph} of~$G$), if it can be obtained by removing vertices of~$G$ (resp.~by removing vertices and edges of~$G$).
A subgraph $H$ of $G$ is \emph{spanning} if $V(H) = V(G)$.
We denote by $G[S]$ the subgraph of $G$ induced by $S$, i.e., obtained by removing every vertex of $V(G) \setminus S$. 

The \emph{degree} of a~vertex $v$ in $G$, denoted by $\deg(v)$, is the number of vertices in $G$ adjacent to~$v$.
The \emph{minimum} (respectively, \emph{maximum}) \emph{degree of $G$} is the minimum (respectively, maximum) degree of a~vertex in $G$.
The \emph{average degree of $G$} is the ratio $\sum_{v \in V(G)} \deg(v)/|V(G)|$.
Recall that the degeneracy of $G$ is the minimum number $k$ such that every induced subgraph of $G$ contains a~vertex of degree at most $k$.
The following are immediate consequences of these definitions: (1) if the minimum degree of $G$ is at least $d$, then the average degree of $G$ and the degeneracy of $G$ is at least $d$; (2) if the degeneracy of $G$ is at least $d$, then $G$ contains an induced subgraph of minimum degree at least $d$.
One more relation between these three parameters is the following folklore result, which is not hard to derive by iteratively removing vertices of degree less than $d/2$.
\begin{observation}\label{obs:avg-min-dgn}
	If the average degree of a~graph $G$ is at least $d$, then $G$ contains an induced subgraph of minimum degree, and thus average degree, at least $\lceil d/2 \rceil$.
\end{observation}  

We use $G-S$ as a~shorthand for $G[V(G) \setminus S]$, and $G-v$, for $G-\{v\}$. We denote by $\text{Aut}(G)$ the automorphism group of a~graph~$G$, and we set $\aut(G) := |\text{Aut}(G)|$. 
If $X$ is a~set of pairwise non-isomorphic $n$-vertex graphs, then the number of labeled graphs isomorphic to a~graph in $X$ is exactly 
\begin{equation}\label{eq:lab-graphs}
		|\lab{X}| = \sum_{G \in X} \frac{n!}{\aut(G)}.
\end{equation}
For any non-negative integer $r$, the \emph{r-subdivision} of a~graph $G$, denoted by $\dt{r}{G}$, is the graph obtained from $G$ by adding one path $P_{uv}$ on $r$ vertices for each edge $uv \in E(G)$, and by replacing each edge $uv \in E(G)$ by a~$(r+1)$-edge path $uP_{uv}v$. In the special case were $r = 0$, $\dt{r}{G}$ is the graph $G$. We will refer to a~vertex from one of the paths $P_{uv}$ as a~\emph{subdivision vertex} and to other vertices as \emph{branching vertices}. 

For completeness, we include a~definition\footnote{Actually, we rather give an established characterization~\cite{DvovrakThesis} as it is more compact.} of classes of \emph{bounded expansion}.
A~class~$\mathcal C$ has \emph{bounded expansion} if there is a~function~$f: \mathbb N \to \mathbb R$ such that for every integer $r \geqslant 0$, no $r$-subdivision of a~graph of average degree larger than~$f(r)$ is a~subgraph of a~graph in~$\mathcal C$.
We will only apply existing results on classes of bounded expansion, and we will do so in a~black-box fashion.
Thus the reader should only observe that these classes are weakly sparse and of bounded average degree (by virtue of satisfying the given property for $r=0$).

\paragraph{Graph Classes.} A~class of graphs is \emph{hereditary} if it is closed under taking induced subgraphs, and it is \emph{monotone} if it closed under taking subgraphs.
A hereditary graph class is \emph{small} if there exists  a~constant $c$ such that the number of $n$-vertex \emph{labeled} graphs in the class is at most $n! \cdot c^n$ for every $n \in \Nn$.

\paragraph{Graph contiguity.} The \emph{contiguity} of a~graph $G$, denoted $\ctg(G)$, is the minimum integer $k$ such that there exists a~linear order of its vertices in which the neighborhood of each vertex of $G$ can be partitioned into at most $k$ disjoint intervals. 
For a~function $k : \Nn \rightarrow \Nn$ and a~class of graphs  $\Cc$, we say that $\Cc$ has contiguity at most $k$, if for every $n$-vertex graph $G \in \Cc$ the contiguity of $G$ is at most~$k(n)$.

The idea of contiguity was introduced in the context of range searching under the name of \emph{crossing number of spanning paths} \cite{Welzl88} (see also \cref{sec:app-paths-low-crossing}).
Independently, it was introduced in the context of DNA reconstruction under the name of \emph{$k$-consecutive ones property} \cite{GGKS95}, which generalizes the notion of \emph{consecutive ones property} for matrices \cite{BL76, FG65}.
More recently, the concept was utilized as part of an FPT algorithm for the FO model-checking problem on monadically stable graph classes \cite{DEMMPT23}, and has appeared in the context of implicit graph representations \cite{Alon23,AAALZ23,BDSZ24}. In particular, Alon used it to show that every hereditary graph class $\Xc$ with $|\Xc_n| = 2^{o(n^2)}$ admits a $n^{1-\varepsilon}$-bit labeling scheme for some $\varepsilon = \varepsilon(\Xc) > 0$ \cite{Alon23}. This was done via the following simple connection between contiguity and adjacency labeling schemes.

\begin{proposition}\label{prop:cont-to-ls}
	Let $k : \Nn \rightarrow \Nn$ be a~function, and $\Cc$ be a~class of graphs of contiguity at most $k$. Then $\Cc$ admits an adjacency labeling scheme of size at most $(2k(n) + 1)\log n$. 
\end{proposition}
\begin{proof}
	Let $G$ be a~$n$-vertex graph  in $\Cc$. 
	To construct adjacency labels for $G$ of the target size, we fix a~linear ordering $\sigma$ of the vertices of $G$ witnessing contiguity at most $k(n)$. The label of a~vertex $v$ of $G$ then consists of its position in $\sigma$, followed by the positions of the endpoints of the at most $k(n)$ intervals describing the neighborhood of $v$.
	The decoder infers adjacency between two vertices $u$ and $v$ by testing if the position of $u$ in $\sigma$ belongs to one of the intervals describing the neighborhood of $v$.  
\end{proof}

\paragraph{Set systems, shatter functions, and neighborhood complexity.} 

A pair $(X, \Sc)$, where $X$ is a~finite set and $\Sc$ is a~family of subsets of $X$ is called a~set system.
The \emph{incidence matrix} $M$ of the set system $(X, \Sc)$ is the matrix with~$|X|$ columns and~$|\Sc|$ rows indexed by the elements of $X$ and~$\Sc$, respectively, where $M_{ij}$ is 1 if the $i$-th set of $\Sc$ contains the $j$-th element of $X$, and 0 otherwise.
	
	The \emph{primal shatter function} of a~set system $(X, \mathcal{S})$ is the function $\pi_{\mathcal{S}}$ given by 
	\[
		\pi_{\mathcal{S}}(m)= \max_{A\subseteq X, \, |A|=m} |\{Y\cap A~:~Y \in \mathcal{S} \}|. 
	\]
	
	The set system $(X^*, \Sc^*)$, where $X^* = \Sc$, $\Sc^* = \{ S_x~:~x \in X\}$, and $S_x = \{S \in \Sc~:~x \in X\}$, is called the dual set system of $(X, \Sc)$.
	The shatter function of $(X^*, \Sc^*)$ is called the \emph{dual shatter function} of $(X, \Sc)$.
    Note that if $M$ is the incidence matrix of $(X,S)$, then the transpose of $M$ is the incidence matrix of $(X^*, \Sc^*)$.
    A~subset $A \subseteq X$ is \emph{shattered} by $\Sc$ if every subset $B$ of $A$ can be obtained as the intersection $B = A \cap Y$ for some $Y \in \Sc$, i.e., $\{ Y \cap A~|~ Y \in \Sc \} = 2^A$.
    The \emph{Vapnik–Chervonenkis dimension} (or \emph{VC dimension} for short) of $\Sc$ is the maximum of the sizes of all shattered subsets of $X$.
	
	Given a~vertex $v$ of a~graph $G$, we use standard notation $N_G(v)$ to denote the set $\{u : vu\in E(G)\}$ of neighbors of $v$ in $G$.
	The \emph{neighborhood set system} of $G$ is the set system  $(V(G), \mathcal{N}_G)$ where  
	\[
		\mathcal{N}_G= \{ N_G(v) : v\in V(G)\}. 
	\]   
	
	Observe that the adjacency matrix of $G$ is the incidence matrix of $(V(G), \mathcal{N}_G)$.
        Thus, since the adjacency matrix of an undirected graph is symmetric, the primal and the dual shatter functions of a~neighborhood set system coincide.
	
	\begin{observation}\label{obs:primal-dual}
		For neighborhood set systems, the primal and dual shatter functions coincide. That is, for any $n$-vertex graph $G$,
		\[
			\pi_{\mathcal{N}_G}(m) = \pi_{\mathcal{N}_G}^{*}(m)\qquad\text{for all }m\in[n]. 
		\]
	\end{observation}
	
	We define the \emph{neighborhood complexity} of a~graph $G$, denoted by $\nu_{G}$, as the primal (equivalently, dual) shatter function of its neighborhood set system, i.e., $\nu_{G}(m) = \pi_{\mathcal{N}_G}(m) = \pi_{\mathcal{N}_G}^*(m)$ for all $m \in [|V(G)|]$.
 
	Given a~graph class $\Cc$, the \emph{neighborhood complexity} of $\Cc$  is the function 
	$\nu_{\Cc}: \mathbb{N} \rightarrow \mathbb{N}$ defined by 
	\[
		\nu_{\Cc}(n) := \sup_{G\in \Cc,\, A\subseteq V(G), \,|A|=n} |\{N(v)\cap A~:~v \in V(G) \}| 
								= \sup_{G\in \Cc} \pi_{\mathcal{N}(G)}(n)
								= \sup_{G\in \Cc} \nu_G(n).
	\]

%%%%%%%%%%%%%%%%%%%%%%%%%%%%%%%%%%% 
\section{Implicit representation for weakly sparse small classes}
\label{sec:weakly-sparse}
%%%%%%%%%%%%%%%%%%%%%%%%%%%%%%%%%%% 

In this section we prove \Cref{thm:small-ws-deg}, which states that any weakly sparse small class has bounded expansion, and  hence, in particular, bounded degeneracy.
Together with the known labeling scheme for bounded degeneracy graphs (see e.g. \cite[Lemma 3.5]{BDSZZ23mono}), this immediately implies that any weakly sparse small class admits an $O(\log n)$-bit adjacency labeling scheme (\cref{thm:ws-ls}).

The proof of \Cref{thm:small-ws-deg} relies on the fact (\cref{cor:DvorakKO}) that for any weakly sparse class $\Xc$ of unbounded expansion there exists an integer $r \geq 1$ such that $\Xc$ contains $r$-subdivisions of graphs of arbitrarily large average degree.
This is a consequence of the forward direction of a characterization of classes of (un)bounded expansion by Dvo{\v{r}}{\'a}k \cite{Dvorak18b} and a classical result due to K{\"{u}}hn and Osthus \cite{Kuhn04}, both stated below.

\begin{theorem}[{\cite[Theorem 5, $(a) \Leftrightarrow (c)$]{Dvorak18b}}]\label{thm:Dvorak}
A~hereditary graph class $\Xc$ has unbounded expansion if and only if there is an integer $r \geqslant 0$ such that for any integer~$d$, $\Xc$ contains the $r$-subdivision of a~graph of average degree at~least~$d$.
\end{theorem}

\begin{theorem}[{\cite[Theorem 2]{Kuhn04}}]\label{thm:KO}
	For all $d,t \in \Nn$, there exists an integer $D := D(d, t)$ such that every $K_{t, t}$-free graph $G$ of average degree at least $D$ contains $\dt{1}{F}$ as an induced subgraph for some graph $F$ of average degree at least $d$.
\end{theorem}

We now extract the announced fact, of which we will only need the \emph{only if} direction.

\begin{corollary}\label{cor:DvorakKO}
A~weakly sparse graph class $\Xc$ is of unbounded expansion if and only if there is an integer $r \geq 1$ such that for any integer~$d$, $\Xc$ contains the $r$-subdivision of a~graph of average degree at~least~$d$.
\end{corollary}
\begin{proof}
  The \emph{if} part follows directly from~\cref{thm:Dvorak}.
  For the \emph{only if} part, we first apply the corresponding implication in~\cref{thm:Dvorak}.
  If $r \geq 1$, we are done.
  If instead $r=0$, then $\Xc$ contains graphs of arbitrarily large average degree.
  Therefore, \cref{thm:KO} implies that $\Xc$ contains 1-subdivisions of graphs of arbitrarily large average degree, i.e., the sought statement holds for $r=1$.
\end{proof}

Equipped with this tool, the proof of \Cref{thm:small-ws-deg} shows the contrapositive as follows.
Assuming a~weakly sparse class $\Xc$ has unbounded expansion, we fix a positive integer $r$ such that $\Xc$ contains $r$-subdivisions of graphs of arbitrarily large average degree, which exists by \cref{cor:DvorakKO}.
Let $F$ be a graph with large average degree such that
the $r$-subdivision of $F$ is in $\Xc$.
First, we extract from $F$ a~subgraph $H$ with large minimum degree that contains a~bounded-degree spanning tree. This is possible due to \cref{obs:avg-min-dgn} and the following lemma.

\begin{lemma}[{\cite[Lemma 3.3]{BDSZZ23mono}}]\label{lem:minDegSubgraphs}
	Let $F$ be a~graph of minimum degree $d$.
	Then $F$ has an induced subgraph $H$ of minimum degree at~least~$d$ with a~spanning tree of \emph{maximum degree} at~most~$d$.
\end{lemma}

\noindent
Next, we use the properties of $H$ to show that it contains many (not necessarily induced) subgraphs each with a~small number of automorphisms, which are translated to many \emph{induced} subgraphs with a~small number of automorphisms in the $r$-subdivision of $H$ and thus in the $r$-subdivision of the original graph $F$. This allows us to lower bound the number of labeled graphs on the same number of vertices in $\Xc$ by counting subgraphs of $H$, and finally conclude that $\Xc$ is not small.
The quantitative connection is made possible by the following observations.

\begin{observation}\label{obs:sub_iso}
	For any integer $r \geq 0$, two graphs $G$ and $H$ are isomorphic if and only if  $\dt{r}{G}$ and $\dt{r}{H}$ are isomorphic.
\end{observation}

\begin{observation}\label{lem:sg2isg}
Let $F$ be a~graph.
For any integer $r > 0$, for every subgraph $H$ of $F$, the graph $\dt{r}{H}$ is isomorphic to an induced subgraph of $\dt{r}{F}$.
\end{observation}

\begin{observation}\label{obs:autosubdivide}
  Let $G$ be a~connected graph containing a~vertex of degree at least $3$, then \[ \aut(\dt{r}{G}) = \aut(G).\]
\end{observation}
\begin{proof}
We start by showing that any automorphism $\varphi$ of $\dt{r}{G}$ maps branching vertices to branching vertices and subdivision vertices to subdivision vertices. First, note that as $G$ has a~vertex of degree at least $3$, the graph $\dt{r}{G}$ also has a~vertex of degree at least $3$ (hence a~branching vertex) say~$u$, which must be mapped by~$\varphi$ to a~vertex of degree at least $3$ (hence, to a~branching vertex), say~$v$.
Now, since $G$ is connected, the branching vertices are exactly the ones that are at distances $0~(\text{mod}~r+1)$ from $u$, and they are mapped to vertices that are at distances $0~(\text{mod}~r+1)$ from~$v$, and thus, to branching vertices.
Hence, the restriction $\tilde{\varphi}$ of $\varphi$ to the branching vertices of $\dt{r}{G}$ is a~permutation of the vertices of $G$.
We claim that $\tilde{\varphi}$ is an automorphism of $G$.
Indeed, two vertices $a$ and $b$ are adjacent in $G$ 
if and only if
$a$ and $b$ are at distance $r+1$ in $\dt{r}{G}$
if and only if 
$\varphi(a)$  and $\varphi(b)$ are at distance $r+1$ in $\dt{r}{G}$
if and only if 
$\tilde{\varphi}(a)=\varphi(a)$ and $\tilde{\varphi}(b)=\varphi(b)$ are adjacent in~$G$.

From this, we deduce that $\varphi \mapsto \tilde{\varphi}$ is a~bijection between the automorphisms of $\dt{r}{G}$ and the automorphisms of $G$: it is injective since an automorphism of $\dt{r}{G}$ is uniquely determined by the images of the branching vertices, and surjective since there are at least as many automorphisms of $\dt{r}{G}$ as there are of $G$.  Hence, $\aut(\dt{r}{G}) = \aut(G)$.
\end{proof}

Before proceeding to our main result, we provide some basic facts about automorphisms, which we will rely on.
Given two graphs $F$ and $G$, we denote by $\sub(F \rightarrow G)$ the number of subgraphs of $G$ isomorphic to $F$, and by $\emb(F \rightarrow G)$ the number of embeddings of $F$ into $G$, i.e.,
the number of injective homomorphisms from $F$ to $G$. Thus
\[\emb(F \rightarrow G) = \sub(F \rightarrow G)\cdot\aut(F).\]
For instance, if $G$ is the cycle on an even number $n$ of vertices, and $F$ is the disjoint union of $n/2$ edges, then $\sub(F \rightarrow G) = 2$ and $\aut(F)=(n/2)! \cdot 2^{\frac{n}{2}}$, so $\emb(F \rightarrow G) = (n/2)! \cdot 2^{\frac{n}{2}+1}$.

We will use the following two well known facts (see~\cite{KLT02}).

\begin{lemma}\label{lem:autGspanF}
Let $F$ be a~spanning subgraph of a~graph $G$. Then 
\[
	\aut(G) \leq \emb(F \rightarrow G) = \sub(F \rightarrow G) \cdot\aut(F).
\]
\end{lemma}

\begin{lemma}\label{lem:autGdeg}
Let $G$ be a~connected graph of maximum degree $\Delta$.
Then 
\[
\aut(G) \leq n \cdot \Delta! \cdot (\Delta-1)^{n-\Delta-1} \leq n \Delta^n.
\]
\end{lemma}

We apply the above two facts to upper bound the number of automorphisms of a~graph containing a~bounded-degree spanning tree in terms its average degree.

\begin{lemma}\label{lem:autTspanG}
Let $d\geq 1$ and $G$ be a~connected $n$-vertex graph with at most $dn$ edges, which has a~spanning tree $T$ of maximum degree $\tau$. Then $\aut(G) \leq   (6\tau d)^{n}$.
\end{lemma}
\begin{proof}
By \Cref{lem:autGdeg}, $\aut(T) \leq n  \tau^n\leq (2\tau)^n$.
Moreover, $\sub(T \rightarrow G)$ is at most the number of ways of choosing $n-1$ edges among $E(G)$. Hence
\[
	\sub(T \rightarrow G) \leqslant \binom{dn}{n-1} \leqslant
	\left( \frac{edn}{n-1} \right)^{n-1} = \left(1+\frac{1}{n-1}\right)^{n-1} \cdot  (ed)^{n-1}   \leq e  (ed)^{n-1} \leq (ed)^{n},
\]
where in the penultimate inequality we used the fact that $1+x \leqslant e^x$ for any real $x$.
Hence, by \Cref{lem:autGspanF}, we have $\aut(G) \leqslant \sub(T \rightarrow G)\cdot \aut(T) \leqslant (ed)^n \cdot (2\tau)^n \leqslant (6\tau d)^{n}$.
\end{proof}

We are now ready to prove the main result of this section.

\smallweakdgen*
\begin{proof}
For the sake of contradiction, suppose that a~weakly sparse graph class~$\Xc$ has unbounded expansion, but there exists a~constant~$c$ such that for every $n \in \Nn$ the number of labeled $n$-vertex graphs in~$\Xc$ is at most $n! \cdot c^n$.
Let $r \geq 1$ be an integer such that $\Xc$ contains $r$-subdivisions of graphs of arbitrary large average degree, which exists by~\Cref{cor:DvorakKO}.

Let $\delta := 24$, $d := \lceil \max (\delta^2, \delta c^{2r}) \rceil$. 
Let $F$ be a graph of average degree at least $4d$, such that $\dt{r}{F}$ is in~$\Xc$, and
let $F'$ be a~subgraph of $F$ of minimum degree at least $2d$, which exists by \cref{obs:avg-min-dgn}.
By \Cref{lem:minDegSubgraphs}, $F'$ contains a~subgraph $H'$ of minimum degree at least $2d$ with a~spanning tree $T$ of maximum degree $2d$.
If $H'$ has more than $d|V(H')|$ edges, we remove all but $d|V(H')| - (|V(H')| - 1)$
of them from the set $E(H') \setminus E(T)$ to obtain a~spanning subgraph with exactly $d|V(H')|$ edges that still contains $T$ as a~spanning tree.
We denote this subgraph by $H$ and let $k := |V(H)|$, and thus $|E(H)| = dk$.

Let $\mathcal{H}$ be the set of spanning subgraphs of $H$ that contain $T$ as a~subgraph and have exactly $\delta k$  edges. Then, since $\delta \leq d$, we have
\begin{equation*}\label{eq:boundonH}
|\mathcal{H}| =
\binom{dk - (k-1)}{ \delta k - (k-1)}   \geqslant \binom{dk - k}{ \delta k - k} \geqslant \left( \frac{d-1}{\delta-1} \right)^{(\delta-1)k}\geqslant \left( \frac{d}{\delta} \right)^{(\delta-1)k}.
\end{equation*}

Since every graph in $\Hc$ has $\delta k \leq dk$
edges and contains a~spanning tree of maximum degree $2d$, from
\cref{lem:autTspanG}, 
\begin{equation}\label{eq:autsofA}
	\max\limits_{A \in \mathcal{H}} \aut(A) \leq (6\cdot 2d\cdot d)^k \leq  (4d)^{2k}. 
\end{equation}  
Thus, denoting by $\Uc$ the set of pairwise non-isomorphic graphs in $\Hc$, we obtain
\begin{equation}\label{eq:sizelower}
	|\mathcal{U}| \geqslant \frac{|\mathcal{H}|}{\max\limits_{A \in \mathcal{H}} \aut(A)} 
	\geqslant
	\left( \frac{d}{\delta} \right)^{(\delta-1)k}  \cdot \frac{1}{(4   d)^{2k}}  
	= \left( \frac{d}{\delta} \right)^{(\delta+1)k}  \cdot \frac{\delta^{2k}}{4^{2k} \cdot d^{4k}} 
	=
	 \left( \frac{d}{\delta} \right)^{(\delta+1)k}  \cdot \frac{6^{2k}}{d^{4k}}  
\end{equation}

Let $\dt{r}{\Uc}$ be the set of $r$-subdivisions of graphs in $\mathcal{U}$,
and observe that every graph in $\dt{r}{\Uc}$ has exactly $N := (r\delta + 1)k$ vertices. 
Since every graph in $\Uc$ is a~subgraph of $F$, and $\dt{r}{F}$ is in $\Xc$, \cref{lem:sg2isg} implies that every graph in $\dt{r}{\Uc}$ is in $\Xc$.
Thus, to contradict the assumption that, for every $n \in \Nn$, class $\Xc$ has at most $n! \cdot c^n$ labeled $n$-vertex graphs,
it is enough to show that $|\lab{\dt{r}{\mathcal{U}}}| > c ^{N}\cdot N!$.
This is what we do in the remainder of the proof.

It follows from \cref{obs:sub_iso} that
$\dt{r}{\Uc}$ is a~set of pairwise non-isomorphic graphs and $|\dt{r}{\Uc}|=|\Uc|$.
Furthermore, since every graph $A$ in $\Hc$ is connected and has $k$ vertices and $\delta k \geq k+1$ edges, it contains a~vertex of degree at least 3 and thus, by \cref{obs:autosubdivide}, $\aut(\dt{r}{A}) = \aut(A)$. Therefore, from \cref{eq:lab-graphs}, we have 
\begin{equation*}%\label{eq:isoclosure}
	|\lab{\dt{r}{\mathcal{U}}}| \geq 
	|\dt{r}{\mathcal{U}}| \cdot \frac{N!}{\max\limits_{A \in \mathcal{H}} \aut(\dt{r}{A})} = |\mathcal{U}| \cdot \frac{N!}{\max\limits_{A \in \mathcal{H}} \aut(A)},
\end{equation*}
and thus, from \eqref{eq:autsofA} and \eqref{eq:sizelower}, and the fact that $N/r = (\delta+1/r)k \leqslant (\delta+1)k$, we obtain
 \begin{equation*}
 	|\lab{\dt{r}{\mathcal{U}}}|
 	\geq  \left(\frac{d}{\delta} \right)^{N/r} \cdot \frac{6^{2k}}{d^{4k}} \cdot \frac{N!}{(4d)^{2k}}
 	> \left(\frac{d}{\delta} \right)^{N/r} \cdot \frac{N!}{ d^{6k}},
 \end{equation*}

Finally, since $\frac{N}{4r} = \frac{(r \delta + 1)k}{4r}  > \frac{\delta k}{4} = 6k $, and
$d$ satisfies both $\sqrt{d} \geq \delta $ and $\left( d/\delta \right)^{1/r} \geq c^2$, we have
\[
	|\lab{\dt{r}{\mathcal{U}}}| 
	>
	\left(\frac{d}{\delta} \right)^{\frac{N}{2r}} \cdot \left(\frac{d}{\delta} \right)^{\frac{N}{2r}}  \cdot \frac{N!}{ d^{6k}} 
	\geq \left(\frac{d}{\delta} \right)^{\frac{N}{2r}}\cdot d^{\frac{N}{4r}}  \cdot \frac{N!}{ d^{6k}}
	> \left( \left(\frac{d}{\delta} \right)^{1/r} \right)^{\frac{N}{2}}\cdot N!
	\geq  c ^{N}\cdot N!,
\]
which completes the proof.
\end{proof}

%%%%%%%%%%%%%%%%%%%%%%%%%%%%%%%%%%% 
\section{Short adjacency labels for small classes}
\label{sec:short-small}
%%%%%%%%%%%%%%%%%%%%%%%%%%%%%%%%%%% 

In this section we prove:

\main*

To do so, we first show a~result of independent interest (in \cref{sec:neighborhood-complexity}), that the neighborhood complexity of any small class is at most $O(n \log n)$. Then, we prove in \cref{sec:neigh-complexity-and-contiguity} that such neighborhood complexity implies $O(\log^2n)$ contiguity. The latter together with \cref{prop:cont-to-ls} implies \cref{thm:overall-main}.
The core part of this proof strategy is to show how the bound on the neighborhood complexity can be translated to the bound on contiguity.
This is done via a~classical result from the theory of set systems of bounded Vapnik-Chervonenkis dimension on so-called \emph{paths with low crossing number} (\cref{th:path-with-low-crossing}) applied to neighborhood set systems.
While the result about paths with low crossing number is a~known fact, we need it in somewhat unusual form that we could not find elsewhere in the literature.
Therefore, in \cref{sec:app-paths-low-crossing} we present a~self-contained proof of this result in a~suitable form. 

%%%%%%%%%%%%%%%%%%%%%%%%%%%%%%%%%%% 
\subsection{Neighborhood complexity of small classes}
\label{sec:neighborhood-complexity}
%%%%%%%%%%%%%%%%%%%%%%%%%%%%%%%%%%% 

In this section we show that small classes have low neighborhood complexity. 

\neighborhoodComplexity*
\begin{proof}
	Let $\Cc$ be a~hereditary small class of graphs and $c \geq 1$ be a~constant such that $|\Cc_n| \leqslant n! \cdot c^n$ holds for every $n \in \mathbb{N}$.
	To prove the claim, we will show that for every graph $H$ in $\Cc$ it holds that
	\[
		\nu_H(m) < 9c^2 m \log (m+1) \text{ \qquad for all } m \in [|V(H)|].
	\]
	Suppose this is not the case. Then, by definition, there exists a~graph $H \in \Cc$ and an $n$-element set $A \subseteq V(H)$ such that 
	$|\{N_H(v)\cap A~:~v \in V(H) \}| \geq 9c^2 n \log (n+1)$.
        Note that the latter inequality implies $n \geq 4$.
	Let $B' \subseteq V(H)$ be any maximum-size set of vertices whose neighborhoods have pairwise distinct intersections with $A$ in the graph $H$, and let $B := B' \setminus A$.
        Denote by $G$ the subgraph of $H$ induced by $A \cup B$.
        By the assumption, we have
	\begin{enumerate} 
		\item $|A| = n$,
		\item $|B| \geq |B'| - |A| \geq 9c^2 n \log (n+1) - n > 8c^2 n \log n$,
		\item $A \cap B = \emptyset$, and 
		\item all vertices in $B$ have pairwise distinct neighborhoods in $A$ in graph $G$.
	\end{enumerate}
	
	Let $n_1 := n \lfloor \log n \rfloor $ and $N := n_1 + n$.
        We will show that the hereditary closure of $G$ contains more than $N!  c^N$ labeled graphs on $N$ vertices, i.e., that $|\text{Lab}(\text{IndSub}_N(G))| > N! c^N$, where $\text{IndSub}_N(G)$ is the set of $N$-vertex induced subgraphs of $G$.
        This would contradict our assumption on the number of graphs in $\Cc$. 
	To do this, let us first fix a~labeling of $A$ using the integers from $1$ to $n$. 
	We claim that any two distinct labeled subsets $X_1,X_2 \subseteq B$ of size $n_1$, labeled with the integers from $n+1$ to~$N$, induce two distinct labeled graphs $G_1=G[A \cup X_1]$ and $G_2=G[A \cup X_2]$ contained in $\mathcal{C}$.
        Indeed, since all vertices in $B$ have pairwise distinct neighborhoods in $A$, $G_1$ and $G_2$ contain a~vertex with the same label in $[n+1,N]$ but different neighborhoods in $A$ (i.e.~different neighborhoods among vertices with labels from $[n]$). 
	
	Hence, the number of $N$-vertex labeled graphs in $\Cc$ that are induced subgraphs of $G$ is lower bounded by the number of ways to choose a~labeled subset of size $n_1$ from $B$, which is
	
	\begin{equation}\label{eq:lower-bound}
		n_1! \cdot \binom{|B|}{n_1} \geqslant 	n_1! \cdot \left(\frac{|B|}{n_1}\right)^{n_1} > n_1! \cdot (8c^2)^{n_1}.
	\end{equation}
	We want to show that this is more than the total number of $N$-vertex labeled graphs in $\Cc$, which, by assumption, does not exceed
	\begin{equation}\label{eq:upper-bound}
		N! \cdot c^N = (n_1+n)! \cdot c^{n_1+n} \leqslant n_1! \cdot (n_1+n)^n \cdot c^{2n_1}.
	\end{equation}
	From \cref{eq:lower-bound} and \cref{eq:upper-bound}, we should then show that 
	\[
		n_1! \cdot (8c^2)^{n_1} \geq n_1! \cdot (n_1+n)^n \cdot c^{2n_1},
	\]
	which is equivalent to establishing that
	\[
		8^{n_1} = 2^{3n_1} \geq (n_1+n)^n = 2^{n \log (n_1+n)}.
	\]
	The latter indeed holds since
	\begin{align*}
		n \log (n_1+n) 
		& \leqslant n \log(2 n \lfloor \log n \rfloor) \\
		& \leqslant n \log 2 + n \log n + n \log \lfloor \log n \rfloor \\
		& \leqslant n(1+\log 2) + n \lfloor \log n \rfloor + n \log \lfloor \log n \rfloor \\
		& \leqslant n\lfloor \log n \rfloor + n \lfloor \log n \rfloor + n \log \lfloor \log n \rfloor \\
		& \leqslant n_1 + n_1 + n_1 = 3n_1,
	\end{align*}
	where the penultimate inequality holds because $n \geq 4$.
\end{proof}

%%%%%%%%%%%%%%%%%%%%%%%%%%%%%%%%%
\subsection{From neighborhood complexity to contiguity}
\label{sec:neigh-complexity-and-contiguity} 
%%%%%%%%%%%%%%%%%%%%%%%%%%%%%%%%%

In this section, we first show how to translate a~bound on neighborhood complexity to a~bound on contiguity (\cref{th:graph-with-low-contiguity}), and then apply it to obtain bounds on contiguity for hereditary small classes (\cref{cor:welzl}).
A crucial technical tool that we use to establish this translation is a~suitable form of a~known result about paths with low crossing number (\cref{th:path-with-low-crossing}). To state this result we must first introduce the required notions. 

Let $(X,\Sc)$ be a~set system.
A~set $S \in \Sc$ \emph{crosses} a~2-element set $\{x,y\} \subset X$ if exactly one of $x$ and $y$ belongs to $S$.
Let $\Fc$ be a~multiset of 2-element subsets in $X$.
The \emph{crossing number of $\Fc$ with respect to a~set $S \in \Sc$} is the number of elements in $\Fc$ that are crossed by $S$.
The \emph{crossing number of $\Fc$ with respect to $\Sc$} is the maximum crossing number of $\Fc$ with respect to a~set $S \in \Sc$.
If $(X,\Fc)$, considered as a~graph, is a~path spanning all elements of $X$, then we say that $(X,\Fc)$ is a~path on $X$.
In this case, the crossing number of $\Fc$ is referred to as the crossing number of the corresponding path.
We are now ready to state the main technical tool whose self-contained proof can be found in \cref{sec:app-paths-low-crossing}.

\begin{restatable}{theorem}{pathcrossing}\label{th:path-with-low-crossing}
	Let $X$ be an $n$-element set, $f : \mathbb{R}_{\geqslant 0} \rightarrow \mathbb{R}_{\geqslant 0}$ be a~strictly increasing function, and $d$ be a~natural number.
	Let $(X, \Sc)$ be a~set system of VC dimension $d$ that satisfies $\pi_{\Sc}^*(m) \leq f(m)$ for all $m \in [n]$.
	Then there exists a~path on $X$ with crossing number at most 
	\[
		2\log |\Sc| + 10d \cdot \sum_{j=1}^{n} \frac{1}{f^{-1}(j/2)}.
	\]
\end{restatable}
 
We will apply \cref{th:path-with-low-crossing} to neighborhood set systems.
Its importance in our context is due to (a) the fact that the dual shatter function of the neighborhood set system of a~graph coincides with its primal dual function (\cref{obs:primal-dual}), i.e., the neighborhood complexity of the graph; and 
(b) the following observation about close relationship between the crossing number of a~path for the neighborhood set system and the contiguity of the graph.

\begin{observation}\label{obs:cross-ctg}
	Let $G$ be an $n$-vertex graph. 
	If there exists a~path on $V(G)$ that has crossing number with respect to $\Nc_G$ at most $k$, then the contiguity of $G$ is at~most~$k/2  + 1$.
\end{observation}
\begin{proof}
	Let $\{ \{v_1,v_2\}, \{v_2,v_3\}, \ldots, \{v_{n-1}, v_n\}\}$ be a~path on $V(G)$ with crossing number with respect to $\Nc_G$ at~most~$k$.
	We claim that the linear ordering~$\sigma = v_1, v_2, \ldots, v_n$ of $V(G)$ witnesses the target bound on the contiguity of $G$.
	
	Let $v$ be an arbitrary vertex in $G$ and let $N$ be its neighborhood.
	For convenience, mark a~vertex of $G$ by 1 if it belongs to $N$ and 0 otherwise.
        This marking together with the ordering~$\sigma$ correspond to the binary word $w = w_1 w_2 \ldots w_n$, where $w_i$ is the mark of $v_i$ for every $i \in [n]$. 
	A~relation between pairs crossed and neighborhoods partitioned now follows from two observations:
	\begin{enumerate}
		\item in the word $w$, the maximal intervals of consecutive 1s correspond to intervals of consecutive vertices in $\sigma$ that partition the neighborhood $N$ of $v$;
		\item each pair of consecutive letters $w_i w_{i+1}$ in $w$ with $w_i \neq w_{i+1}$ corresponds to a~crossing of $\{v_i,v_{i+1}\}$ by $N$.
	\end{enumerate} To conclude, each interval in the partition of $N$ corresponds to two crossings apart from the at most two intervals at each end of the ordering which each corresponds to one crossing. Thus the number of intervals is at most $2 + \lfloor \frac{k-2}{2} \rfloor \leq k/2 +1$, as claimed.    
\end{proof}

The next theorem translates a~bound on neighborhood complexity to a~bound on contiguity. It follows directly from \cref{th:path-with-low-crossing}, \cref{obs:cross-ctg}, \cref{obs:primal-dual}, and the definition of neighborhood complexity.

\begin{theorem}\label{th:graph-with-low-contiguity}
	Let $G$ be an $n$-vertex graph, $f : \mathbb{R}_{\geqslant 0} \rightarrow \mathbb{R}_{\geqslant 0}$ be a~strictly increasing function, and $d$ be a~natural number.
	If $(V,\Nc_G)$ has VC dimension $d$ and $\nu_G(m) \leq f(m)$ for all $m \in [n]$, then  
	\[
		\ctg(G) \leq 1+ \log n + 5d \cdot \sum_{j=1}^{n} \frac{1}{f^{-1}(j/2)}.
	\]
\end{theorem}

We now apply \cref{th:graph-with-low-contiguity} to hereditary small classes, which have $O(n\log n)$ neighborhood complexity by \cref{thm:neighborhood-complexity}, to obtain a~bound on their contiguity in closed form.  

\begin{theorem}\label{cor:welzl}
	Let $\Xc$ be a~hereditary small class, then the contiguity of $\Xc$ is $ O(\log^2 n)$.
\end{theorem} 
\begin{proof}
  Since $\mathcal{X}$ is a~small class, by \cref{thm:neighborhood-complexity}  there exists some constant $C$ such that $\nu_{\mathcal{X}}\leq f(n)$ where $f(n):=Cn\log n$.
  Consequently, for every $n\in \mathbb{N}$ and every $n$-vertex graph $G\in \mathcal{X}$, we have $\pi_{\mathcal{N}_G}(m)\leq f(m)$ for all $m\in [n]$.
  Thus, by definition, the VC dimension $d$ of $(V(G), \mathcal{N}_G)$
  satisfies $2^d \leq Cd \log d$, which implies 
  \[
  		d \leq \log C + \log d + \log \log d \leq \log C + 4d/5,
  \]
  and therefore $d \leq 4 \log C$.
	 
	Now, by \cref{th:graph-with-low-contiguity}, we have
	\begin{equation}\label{eq:thm4.3contents}
		\ctg(G) \leq 1 + \log n + 20 \cdot \log C \cdot \sum_{j=1}^{n} \frac{1}{f^{-1}(j)}.
	\end{equation}
	Since $f$ is increasing on $[1,\infty)$ with image $[0, \infty)$, its inverse $f^{-1}$ exists on $[0, \infty)$, and is also increasing.
    The exact inverse $f^{-1}$ is somewhat complicated, in particular $f^{-1}(1)=e^{W(\ln(2)/C)}$, where $W(\cdot)$ is the Lambert function.
    We consider $f^{-1}(1)>0$ just as a~constant depending on $C$, and show that \begin{equation}\label{eq:inversebdd} f^{-1}(x) \geq \frac{x}{C\log x} \qquad  \text{ for all }x\geq 2.\end{equation}  
	To show \eqref{eq:inversebdd},  since $\log x \geq 1$ for all $x\geq 2$ and we can assume $C\geq 1$, we have 
	\begin{equation}\label{eq:bound4inverse} f\left(\frac{x}{C\log x}\right) = C\cdot \frac{x}{C\log x } \cdot \log \left(\frac{x}{C\log x}\right) \leq x. \end{equation} 
 	Thus, applying $f^{-1}$ to both sides of \eqref{eq:bound4inverse} gives $ f^{-1}(x) \geq \frac{x}{C\log x}$, establishing  \eqref{eq:inversebdd}.

 	Then, by  \eqref{eq:thm4.3contents}, 
	\begin{align*}
		\ctg(G) &\leq 
		1+ \log n + 20 \cdot \log C \cdot \sum_{j=1}^{n} \frac{1}{f^{-1}(j)} \\ &\leq 
		1+ \log n + 20\cdot \log C\cdot e^{-W(\ln(2)/C)} + 20 \cdot \log C\cdot \sum_{j=2}^{n} \frac{C\log j}{j} \\
		&	= 	\Theta( \log^2 n),
	\end{align*}as claimed. 
\end{proof}

The main \cref{thm:overall-main} of this section is a~consequence of~\cref{thm:neighborhood-complexity,cor:welzl,prop:cont-to-ls}.

%%%%%%%%%%%%%%%%%%%%%%%%%%%%%%%%%
\section{Conclusion}
%%%%%%%%%%%%%%%%%%%%%%%%%%%%%%%%%

In this paper we obtained strong evidence in support of the Small Implicit Graph Conjecture that posits that every hereditary small class admits an $O(\log n)$-bit adjacency labeling scheme. Specifically, we showed that (1) every weakly sparse small class admits such a labeling scheme; (2) every hereditary small class admits an $O(\log^3 n)$-bit adjacency labeling scheme. 
To obtain these results we established two structural properties of hereditary small classes that are of independent interest. 
The first property is that every weakly sparse small class has bounded degeneracy, and even, bounded expansion. 
The second property is that every hereditary small class has neighborhood complexity $O(n \log n)$.

The latter property leaves a tantalizing open question of whether the bound on the neighborhood complexity can be further improved.
All hereditary small classes known to us have linear neighborhood complexity (e.g., classes of bounded twin-width \cite{BKRTW22,BFLP24}), which motivates the following:

\begin{conjecture}\label{conj:small-linear-neigh-complexity}
	Every hereditary small class of graphs has neighborhood complexity $O(n)$.
\end{conjecture}

\noindent
If this conjecture is true, our approach would imply $O(\log^2 n)$-bit labeling schemes for all hereditary small classes.

\iftoggle{anonymous}{% ANONYMOUS
}{%
	\bigskip
	\bigskip
	
	\noindent
	\textbf{Acknowledgments.}
	We are grateful to Maksim Zhukovskii for many stimulating discussions on the topic of this paper.
        The first author thanks Szymon Toruńczyk for introducing him to (and answering questions on) the topic of paths with low crossing number.
	This work has been supported by the Royal Society (IES\textbackslash R1\textbackslash 231083), by the ANR projects TWIN-WIDTH (ANR-21-CE48-0014) and Digraphs (ANR-19-CE48-0013), and also the EPSRC project EP/T004878/1: Multilayer Algorithmics to Leverage Graph Structure.
}

\bibliographystyle{alpha}
\bibliography{biblio}

%%%%%%%%%%%%%%%%%%%%%%%%%%%%%%%%
\appendix
%%%%%%%%%%%%%%%%%%%%%%%%%%%%%%%%

%%%%%%%%%%%%%%%%%%%%%%%%%%%%%%%%
\section{Paths with low crossing number}\label{sec:app-paths-low-crossing}
%%%%%%%%%%%%%%%%%%%%%%%%%%%%%%%%
The main goal of this section is to provide a~self-contained proof of \cref{th:path-with-low-crossing}. 
To state the theorem, we recall definitions of some crucial notions.
Given a~set system $(X,\Sc)$, we say that a~set $S \in \Sc$ \emph{crosses} a~2-element set $\{x,y\} \subset X$ if exactly one of $x$ and $y$ belongs to $S$.
Let $\Fc$ be a~multiset of 2-element subsets in $X$. The crossing number of $\Fc$ with respect to a~set $S \in \Sc$ is the number of elements in $\Fc$ that are crossed by $S$. The crossing number of $\Fc$ with respect to $\Sc$ is the maximum crossing number of $\Fc$ with respect to a~set $S \in \Sc$.
If $(X,\Fc)$, considered as a~graph, is a~tree or a~path spanning all elements of $X$, then we say that 
$(X,\Fc)$ is a~tree or a~path on $X$, respectively. In this case, the crossing number of $\Fc$ is referred to as the crossing number of the corresponding tree or path, respectively.

\pathcrossing*

This theorem is a~known result stated in an unusual form, but one that is useful for us.
The main feature of the above formulation is that the bound on the dual shatter function is given in a~general form (expressed as a~function $f$) rather than polynomial; in turn, the bound on the crossing number is expressed via the inverse of~$f$.
More specifically, the above formulation has two main differences compared to the standard one (for instance \cite{MatGeoDisc,Welzl88}). 
First, in a~standard formulation, the bound on the dual shatter function is $\pi_{\Sc}^*(m) \leq m^r$ for some constant $r>1$.
Second, the resulting bound on the crossing number of a~path on $X$ is in the form of $O(n^{1-1/r})$. 
Note that if we take $f(m)$ to be  $m^r$ in \cref{th:path-with-low-crossing}, then $f^{-1}(m) = m^{1/r}$ and \[2\log |\Sc| + 10d \cdot \sum_{j=1}^{n} \frac{1}{(j/2)^{1/r}}\] can be upper bounded by $O(n^{1-1/r})$, which recovers the result in a~standard form. 

The benefit of the present formulation is that one can apply \cref{th:path-with-low-crossing} when a~preferable upper bound on $\pi_{\Sc}^*(m)$ is non-polynomial. For example, if $f(m) = m \cdot g(m)$ for some non-decreasing function $g(m)$, then \cref{th:path-with-low-crossing} gives a~bound of $O(g(n) \cdot \log n)$ on the crossing number. In particular, this becomes useful when $g(m) = m^{o(1)}$; e.g., if $g(m) = \log(m)$, the bound on the crossing number becomes $O(\log^2 n)$.

To provide a~self-contained proof of \cref{th:path-with-low-crossing},
we need to state a~number of auxiliary known results in a~suitable (non-standard) form and 
repeat their proofs with some adjustments. Before each statement we explain how its formulation and proof differs from standard ones. We hope this exposition might be of use in some other work. 

\paragraph{Packing lemma.}
We begin with Haussler's Packing Lemma \cite{Haussler95} (\cref{lem:packing}), which is an improvement over a~quantitatively weaker Packing Lemma obtained by Dudley \cite{Dudley} and re-discovered by Welzl \cite{Welzl88}. The proof below is a~simplification by Chazelle of Haussler's proof \cite{Haussler95}, which is taken from Matou{\v{s}}ek's  book \cite{MatGeoDisc}.

A set system $(X,\Sc)$ is \emph{$\delta$-separated}
if for any two distinct sets $S_1, S_2 \in \Sc$ their symmetric difference contains at least $\delta$ elements, i.e., $|(S_1 \setminus S_2) \cup (S_2 \setminus S_1)| \geq \delta$.
In a~standard formulation, the lemma assumes that the shatter function $\pi_{\Sc}(m)$ is bounded by $O(m^r)$ for some $r$, and the bound on the cardinality of a~$\delta$-separated set $\Pc$ is $O((n/\delta)^r)$.
For our purposes, in \cref{lem:packing}, we state the bound on $|\Pc|$ directly in terms of the shatter function.
Note that the standard form of the lemma follows from ours if one uses the bound $\pi_{\Sc}(m) = O(m^r)$.
The difference of the proof below from the one in \cite{MatGeoDisc} is that in \cref{eq:EW1}
we bound $t$ in terms of the shatter function directly, rather than in terms of the bound on the shatter function.
 
\begin{lemma}[Packing Lemma, see~{\cite[Lemma 5.14]{MatGeoDisc}}]\label{lem:packing}
  Let $(X,\mathcal{S})$  be a~set system on an \mbox{$n$-element} set of VC dimension $d$.
  Let $\delta\in [n]$ be an integer, and let $\mathcal{P}\subseteq \mathcal{S}$ be $\delta$-separated.
  Then, 
	\[
		|\mathcal{P}|\leq 2\cdot \pi_{\mathcal{S}}\left(\left\lceil \frac{4dn}{\delta}\right\rceil \right).
	\]
\end{lemma}

Before we prove this lemma we must introduce an auxiliary graph construction. For a~set system $(X, \mathcal{S})$, we define the \emph{unit distance graph} $\operatorname{UD}(\mathcal{S})$ to be the graph with  \[V(\operatorname{UD}(\mathcal{S})) =\mathcal{S}\qquad  \text{and}\qquad E(\operatorname{UD}(\mathcal{S})) = \{ \{S, S'\} :  |S \Delta S'| = 1\}.\]  
We will also need the following bound on the number of its edges. 

\begin{lemma}[{\cite[Lemma 5.15]{MatGeoDisc}}]\label{lem:edgesUD} 
	If $(X,\mathcal{S})$ is a~set system of VC dimension $d$ on a~finite set~$X$ then
	the unit distance graph $\operatorname{UD}(\mathcal{S})$ has at most $d |\mathcal{S}|$ edges.
\end{lemma} 
 
 With this, we are now ready to prove the Packing Lemma.
\begin{proof}[Proof of \cref{lem:packing}] 
	Choose a~random $s$-element subset $A \subseteq  X$, where $s$ is given by  
        \begin{equation}\label{eq:sdef} s :=\left\lceil \frac{ 4d n}{ \delta} \right\rceil.   \end{equation}
        Set $\mathcal{Q} := \{S\cap A~:~S \in \mathcal{P} \}$, and for each set $Q \in \mathcal{Q}$ define its weight $w(Q)$ as the number of sets $S \in \mathcal{P}$ with $S \cap A~= Q$. Note that \[\sum_{Q\in \mathcal{Q}} w(Q) = |\mathcal{P}|.\]

Let $E:=E(\operatorname{UD}(\mathcal{Q}))$ be the edge set of the unit distance graph $\operatorname{UD}(\mathcal{Q})$, and define an edge weighting of $\operatorname{UD}(\mathcal{Q})$ by giving each edge $e = \{Q, Q'\}\in E$ weight $w(e) := \min(w(Q), w(Q'))$. Let \[W = \sum_{e\in E} w(e), \] where we note that this random variable depends on the random choice of $A$. The desired bound on $|\mathcal{P}|$ in the Packing Lemma is obtained by estimating the expectation of $W$ in two different ways.

First, we claim that for any set $A \subseteq  X$, we have  the following (deterministic) bound on $W$ 
\begin{equation}\label{eq:firstbdd}
	W \leq 2 d\cdot \sum_{Q\in \mathcal{Q}} w(Q)  =2 d\cdot|\mathcal{P}|.
\end{equation}
To see this, note first that the VC dimension of $\mathcal{P}$ is at most $d$, and thus, by \cref{lem:edgesUD}, the unit distance graph $\operatorname{UD}(\mathcal{Q})$ has some vertex $Q_0$ of degree at most $2d$. By removing $Q_0$,  the total edge weight drops by at most $2d \cdot w(Q_0 )$, and we are left with a~unit distance graph on $(X,\mathcal{Q}')$, where $\mathcal{Q}' =\mathcal{Q} \setminus Q_0$ and $\mathcal{Q}'$ has VC dimension at most $d$. We can thus repeat this process until no vertices are left, and we see that the sum of edge weights removed is at most $2d$ times the sum of vertex weights, as claimed.

Next, we bound the expectation $\Ex{W}$ from below by considering the following
random experiment. First, we choose a~random $(s - 1)$-element set $A \subseteq  X$, and then we choose a~random element $a \in X \setminus A'$. The set $A = A' \cup  \{a\}$ is a~uniform random $s$-element subset of $X$,  we consider $\operatorname{UD}(\mathcal{Q})$ with the same edge weighting as above. Each edge of $\operatorname{UD}(\mathcal{Q})$ is a~pair of sets of $\mathcal{Q}$ differing in exactly one element of $A$. We let $E_1 \subseteq  E$ be the edges for which the difference element is $a$, and let $W_1$ be the sum of their weights. By symmetry, we have \begin{equation}\label{eq:eqs1}\Ex{W} = s \cdot  \Ex{W_1}.\end{equation}

We are going to bound $\Ex{W_1}$ from below.
Let $A'  \subset  X$ be an arbitrary but fixed $(s -1)$-element set.
We estimate the conditional expectation $\Ex{W_1 \mid A' }$; 
that is, the expected value of $W_1$ when conditioned on a~fixed set $A'$, and $a \in X\setminus A'$ is selected uniformly at random.
Divide the sets of $\mathcal{P}$ into equivalence classes $\mathcal{P}_1 , \mathcal{P}_2 ,\dots , \mathcal{P}_t$ according to their intersection with the set $A'$. By the definition of $\pi_{\mathcal{S}}(x)$ and since it is non-decreasing, we have  \begin{equation}\label{eq:tbound}t \leq  \pi_{\mathcal{S}}(s -1) \leq \pi_{\mathcal{S}}(s). \end{equation}

Let $\mathcal{P}_i$ be one of the equivalence classes, and $b := |\mathcal{P}_i |$.  Suppose that an element $a \in  X \setminus A'$ has been chosen. If $b_1$ sets of $\mathcal{P}_i$ contain $a$ and $b_2 = b - b_1$ sets do not contain $a$, then the class $\mathcal{P}_i$ gives rise
to an edge of $E_1$ of weight $\min(b_1 , b_2)$. Note that $b_1$ and $b_2$ depend on the choice of $a$ while $b$ does not.

For any non-negative real numbers $b_1 , b_2$ with $b_1 + b_2 = b$, we have $\min(b_1 , b_2 ) \geq  b_1 b_2 /b$.
The value $b_1  b_2$ is the number of ordered pairs of sets $(S_1 , S_2)$ with $S_1 , S_2$ being two sets from the class $\mathcal{P}_i$ where one contains $a$ and the other one does not.
Now, since $\mathcal{P}$ is $\delta$-separated by hypothesis, if $S_1, S_2 \in \mathcal{P}_i$ are two distinct sets, then they differ in at least $\delta$ elements.
Therefore, the probability that $S_1$ and $S_2$ differ in a~random element $a \in X \setminus A'$ is at least $ \frac{\delta}{n -s+1} \geq \delta / n$.
Hence the expected contribution of each pair $(S_1 , S_2 )$ of distinct sets of $\mathcal{P}_i$ to the quantity $b_1 b_2$ is at least $\delta/n$, thus 
\[
	\Ex{b_1 b_2 \mid A'} \geq  b(b - 1)\cdot \frac{\delta}{n}.
\]
Consequently,  the expected contribution $\Ex{\min(b_1 , b_2 ) \mid A' }$ of the equivalence class $\mathcal{P}_i$ to the sum of edge weights $W_1$ is at least $\Ex{b_1 b_2/b \mid A'} \geq (b -1)\delta /n$.
Summing up over all equivalence classes, and using the bound on $t$ from \eqref{eq:tbound},  

\begin{equation}\label{eq:EW1}
	\Ex{W_1} \geq \Ex{W_1 \mid A' } \geq  \frac{\delta }{n}\sum_{i=1}^t \left(|\mathcal{P}_i| -1 \right) \geq   \frac{\delta }{n}\left(|\mathcal{P}| -t \right) \geq  \frac{\delta }{n}\left(|\mathcal{P}| - \pi_{\mathcal{S}}(s) \right). 
\end{equation}Combining this with the estimate \eqref{eq:firstbdd}, the equality \eqref{eq:eqs1}, and the definition \eqref{eq:sdef} of $s$, leads to 

\[
	2d\cdot |\mathcal{P}| \geq 
	s\cdot \frac{\delta }{n }\cdot \left(|\mathcal{P}| - \pi_{\mathcal{S}}(s) \right) 
	\geq  4d\cdot \left(|\mathcal{P}| - \pi_{\mathcal{S}}(s) \right) .   
\]   
Rearranging  gives $|\mathcal{P}| \leq 2\cdot \pi_{\mathcal{S}}(s)$, so the statement follows from \eqref{eq:sdef}. 
\end{proof}

\paragraph{Short Edge Lemma.}

A standard form of the Short Edge Lemma assumes a~polynomial bound on the dual shatter function, i.e., $\pi^*_{\Sc}(m) = O(m^r)$ for some constant $r>1$,
and the resulting bound is expressed via the bound $O(m^{1/r})$ on the inverse of $\pi^*_{\Sc}(m)$.
In our presentation of the Short Edge Lemma (\cref {lem:short-edge}) below we use a~bound on the dual shatter function in a~general from expressed as a~function $f$, and the resulting bound is stated in terms of $f^{-1}$, the inverse of $f$. As before, by plugging in the bound $\pi^*_{\Sc}(m) = O(m^r)$ in \cref {lem:short-edge} one would recover the standard form of the lemma.
The proof is taken from \cite{MatGeoDisc} with suitable adjustments. 

\begin{lemma}[Short Edge Lemma, see~{\cite[Lemma 5.18]{MatGeoDisc}}]\label{lem:short-edge}
	Let $X$ be an $n$-element set, $f : \mathbb{R}_{\geqslant 0} \rightarrow \mathbb{R}_{\geqslant 0}$ be a~strictly increasing function, and $d$ be a~natural number.
	Let $(X, \Sc)$ be a~set system of VC dimension $d$ with $\pi_{\Sc}^*(m) \leq f(m)$ for all $m \in [n]$.
	Then,  for any multiset $\Qc$ with elements from $\Sc$, there exist elements $x,y\in X$ such that the edge $\{x,y\}$ is crossed by at most 
	\[
	5d \cdot \frac{|\Qc|}{f^{-1}(n/2)} 
	\]
	sets of $\Qc$.
\end{lemma} 
\begin{proof}
	We form a~set system $\Dc$ which is dual to $\Qc$. That is, we consider the multiset $\Qc$ as the ground set (if some set appears in $\Qc$ several times, then it is considered with the appropriate multiplicity). For each $x \in X$, we let $D_x$ be the set of all sets of $\Qc$ containing $x$, and we set 
	\[
	\Dc=\{D_x : x\in X \}.
	\] 
	The symmetric difference $D_x\Delta D_y$ of two sets from $\Dc$ consists of the sets in $\Qc$ crossing the pair $\{x,y\}$. We thus want to show that $D_x\Delta D_y$ is small for some $x\neq y$. We may assume $D_x\neq D_y$ for $x\neq y$, as otherwise we are done, and hence $|\Dc|=|X|=n$. 
	
	The primal shatter function $\pi_{\Dc}$ is certainly no larger than the dual shatter function $\pi_{\Sc}^{*}$, and hence $\pi_{\Dc}(m) \leq f(m)$ by the assumption on $\pi_{\Sc}^{*}$. Suppose that any two sets $D_x,D_y \in \Dc$ have symmetric difference at least $\delta$, then the packing lemma (\cref{lem:packing}) implies 
	
	\[
	n = |\mathcal{D}| 
	\leq 2\cdot \pi_{\mathcal{S}}^*\left(\left\lceil \frac{4d|\Qc|}{\delta}\right\rceil \right)
	\leq 2f\left(\left\lceil \frac{4d |\Qc|}{\delta}\right\rceil \right).
	\] 
	Since $f$ is strictly increasing, $f^{-1}$ is also strictly increasing. Therefore, we have 
	\[
	f^{-1}(n/2) \leq \left\lceil \frac{4d |\Qc|}{\delta}\right\rceil \leq \frac{5d |\Qc|}{\delta},
	\]
	and therefore 
	\[
	\delta \leq 5d \cdot \frac{|\Qc|}{f^{-1}(n/2)},
	\]
	as claimed.
\end{proof}

\paragraph{Trees with low crossing number.}

As with the Short Edge Lemma, the following result  (\cref{th:tree-with-low-crossing}) is stated with a~general bound $f$ on the dual shatter function and the resulting bound on the crossing number of a~tree is expressed in terms of $f^{-1}$, the inverse of $f$.
The proof is taken from \cite{MatGeoDisc}, which is very similar to Welzl's original proof \cite{Welzl88}.
The statement in \cite{MatGeoDisc} is for crossing number of matchings, and therefore it is slightly modified to work for trees as in \cite{Welzl88}.

\begin{theorem}[see~{\cite[Lemma 5.17]{MatGeoDisc}}]\label{th:tree-with-low-crossing}
	Let $X$ be an $n$-element set, $f : \mathbb{R}_{\geqslant 0} \rightarrow \mathbb{R}_{\geqslant 0}$ be a~strictly increasing function, and $d$ be a~natural number.
	Let $(X, \Sc)$ be a~set system of VC dimension $d$ with $\pi_{\Sc}^*(m) \leq f(m)$ for all $m \in [n]$.
	Then there exists a~tree $T$ on $X$ with crossing number at most 
	\[
	\log |\Sc| + 5d \cdot \sum_{j=1}^{n} \frac{1}{f^{-1}(j/2)}.
	\]
\end{theorem}
\begin{proof}
	We build $T$ by selecting edges one by one according to  the following strategy.
	\begin{enumerate}
		\item Suppose $\{u_1,v_1\}, \ldots, \{u_i,v_i\}$ have already been selected.
		\item\label{itm:haus2} Define the weight $w_i(S)$ of a~set $S \in \Sc$ as $2^{k_i(S)}$, where
		$k_i(S)$ is the number of edges among $\{u_1,v_1\}, \ldots, \{u_i,v_i\}$ crossed by $S$. In particular, $w_0(S) = 1$ for all $S \in \Sc$.
		\item\label{itm:haus3} We select the next edge $\{u_{i+1},v_{i+1}\}$ from $X_i = X \setminus \{u_1, u_2, \ldots, u_i\}$ as a~pair of points with the total weight of sets crossing $\{u_{i+1}, v_{i+1}\}$ being the smallest possible.
		\item We continue in this manner until $n-1$ edges have been selected.
	\end{enumerate}
	
	Let $E := \{ \{u_1,v_1\}, \{u_2,v_2\} \ldots, \{u_{n-1},v_{n-1}\} \}$. Then $T=(X,E)$ is a~tree.
	Next, we will bound the crossing number $k$ of the resulting tree $T$. To this end, we estimate the final total weight of all sets of $\Sc$, i.e.,
	\[
	w_{n-1}(\Sc) = \sum_{S \in \Sc} w_{n-1}(S).
	\]
	
	By definition of $w_{n-1}$ we have $k \leq \log w_{n-1}(\Sc)$.
	
	Let us investigate how $w_{i+1}(\Sc)$ increases compared to $w_i(\Sc)$.
	Let $\Sc_{i+1}$ denote the collection of the sets of $\Sc$ crossing $\{u_{i+1},v_{i+1}\}$.
	For each set in $\Sc_{i+1}$ the weight increases by a~multiplicative factor of two, and for the other sets it remains unchanged. 
	From this we get
	\[
		w_{i+1}(\Sc) 
		= w_i(\Sc) - w_i(\Sc_{i+1}) + 2w_{i}(\Sc_{i+1}) 
		= w_i(\Sc) \left[ 1 + \frac{w_i(\Sc_{i+1})}{w_i(\Sc)} \right].
	\]
	
	We now estimate $\frac{w_i(\Sc_{i+1})}{w_i(\Sc)}$ using the Short Edge Lemma (\cref{lem:short-edge}).
 	To begin we define a~multiset $\Qc_i$ of sets by adding $S \cap X_i$ to $\Qc_i$
	with multiplicity $w_i(S)$ for every $S \in \Sc$. 
	Thus 
	\[
		|\mathcal{Q}_i| = \sum_{S\in \mathcal{S} }w_i(S) = w_i(\Sc).
	\]
	Note that, by the construction of $\Qc_i$, the number of sets in $\Qc_i$ that cross $\{u_{i+1}, v_{i+1}\}$ is $w_i(\Sc_{i+1})$.
        Furthermore, by our choice of edge in \eqref{itm:haus2} and \eqref{itm:haus3},  the edge $\{u_{i+1}, v_{i+1}\}$ is crossed by the minimum number of sets in $\mathcal{Q}_i$.
        Thus, by applying the Short Edge Lemma on the set system 
	$(X_i, \{ S\cap X_i~:~S \in \Sc\})$ to the multiset $\Qc_i$, and defining $n_i := |X_i| = n-i$, we obtain:
	\[
		w_i(\Sc_{i+1}) \leq 5d \cdot \frac{w_i(\Sc)}{f^{-1}(n_i/2)}.
	\]
	Hence, we have
	\[
		w_{i+1}(\Sc) \leq w_i(\Sc) \left[ 1 + \frac{5d}{f^{-1}((n-i)/2)}\right].
	\] Therefore
	\[
	w_{n-1} \leq 
	w_0(\Sc) \cdot \prod_{i=0}^{n-1} \left[ 1 + \frac{5d}{f^{-1}((n-i)/2)}\right]
	= |\Sc| \cdot \prod_{j=1}^{n} \left[ 1 + \frac{5d}{f^{-1}(j/2)}\right].
	\]Taking logarithms of both sides and using the inequality $\ln(1+x) \leq x$ we obtain
	\begin{align*}
		k & \leq \log w_{n-1}(\Sc) \leq \log |\Sc| + 5d \cdot \sum_{j=1}^{n} \frac{1}{f^{-1}(j/2)},
	\end{align*}as claimed.
\end{proof}

\paragraph{From a~tree to a~path.}
Finally, \cref{th:path-with-low-crossing} is obtained from \cref{th:tree-with-low-crossing} by using the following \cref{lem:from-tree-to-path} to ``convert'' a~tree with crossing number $k$ to a~path with crossing number at most $2k$. The argument is taken from \cite{Welzl88} and we provided it here for completeness.

\begin{lemma}[{\cite[Lemma 3.3]{Welzl88}}]\label{lem:from-tree-to-path}
	Let $(X, \Sc)$ be a~set system, and $T$ be a~tree on $X$ with crossing number at most $k$ with respect to $\mathcal{S}$.  Then, there exists a~path on $X$ with crossing number at most $2k$ with respect to $\Sc$. 
\end{lemma} 
\begin{proof} We begin with the following claim. 
	\begin{claim}\label{clm:edgeremoval}
		Given any family $\mathcal{F}$ of two element sets on $X$,  replacing any pair $\{x,y \}, \{y,z\}\in \mathcal{F}$ by the set $\{x,z\}$ will not increase the crossing number with respect to $\Sc$. 
	\end{claim}
\begin{poc}
  If any $S\in \mathcal{S}$ crosses $\{x,z\}$, then exactly one of  $x$ or $y$ is not contained in $S$.
  Thus, $S$ must cross one of $\{x,y\}$ or $\{y,z\}$. % (which of these sets is crossed depends on whether $y\in S$).   
\end{poc}
Let $\mathcal{D}= \{\{u_1, u_2\}, \{u_2, u_3 \}, \dots , \{u_{\ell-1}, u_\ell \}\}$ be a~multiset consisting of edges from $T$ in the order they are traversed in a~depth-first search (DFS) tour\footnote{A DFS tour visits all vertices of the tree starting from a given vertex and going as far as possible down a given branch, then backtracking until it finds an unexplored path.} of~$T$ starting from an arbitrary vertex.
Additionally, let  $x_1, \dots, x_n$ be a~labeling of $X$ in the order they are discovered by the DFS tour $\mathcal{D}$. 
	
	Note that each edge of $T$ appears exactly twice in $\mathcal{D}$, thus the crossing number of $\mathcal{D}$ is at most $2k$. We can now apply \cref{clm:edgeremoval} iteratively to reduce $\mathcal{D}$ to $\{\{x_1,x_2\}, \{x_2,x_3\}, \dots,$ $\{x_{n-1}, x_n \}\}$ without increasing the crossing number with respect to $\Sc$.    
\end{proof}
 
\end{document}